\documentclass[10pt; a4paper]{amsart} 
\usepackage{amscd, amsfonts, amsmath, amssymb, amsthm, fixmath, mathrsfs}
\usepackage[all]{xy} 
\usepackage[perpage,symbol]{footmisc} 
\setfnsymbol{wiley}

\makeindex

\theoremstyle{definition} 
\newtheorem{Unity}{Unity}[section] 
\newtheorem*{Definition*}{Definition} 
\newtheorem{Definition}[Unity]{Definition} 

\theoremstyle{plain} 
\newtheorem*{Theorem*}{Theorem}
\newtheorem{Theorem}[Unity]{Theorem}
\newtheorem{Proposition}[Unity]{Proposition}
\newtheorem{Corollary}[Unity]{Corollary}
\newtheorem{Lemma}[Unity]{Lemma}
\newtheorem*{Conjecture*}{Conjecture}

\theoremstyle{remark} 
\newtheorem*{Remark*}{Remark}
\newtheorem{Remark}[Unity]{Remark}

\numberwithin{Unity}{section}

\newcommand{\Z}{\mathbb{Z}}

\newcommand{\rk}{\mathrm{rk}}

\newcommand{\im}{\mathrm{Im}}

\newcommand{\Omg}{\mathrm{\Omega}}

\newcommand{\Ox}{\mathscr{O}}

\begin{document}

\title{On a Conjecture of Lan-Sheng-Zuo on Semistable Higgs Bundles: Rank $3$ Case}
\author{Lingguang Li}
\address{Department of Mathematics, Tongji University, Shanghai, P. R. China\\
School of Mathematical Sciences, Fudan University, Shanghai, P. R. China}
\email{LG.Lee@amss.ac.cn}
\thanks{This work is supported by the SFB/TR 45 'Periods, Moduli Spaces and Arithmetic of Algebraic Varieties' of the DFG, and partially supported by the National Natural Science Foundation (No. 11271275).}
\footnotetext[1]{Very recently, A. Langer \cite{Langer13i} and independently Lan-Sheng-Yang-Zuo \cite{LanShengZuo13i} has proven the conjecture for ranks less than or equal to $p$ case.}
\begin{abstract}
Let $X$ be a smooth projective curve of genus $g$ over an algebraically closed field $k$ of characteristic $p>2$. We prove that any rank $3$ nilpotent semistable Higgs bundle $(E,\theta)$ on $X$ is a strongly semistable Higgs bundle. This gives a partially affirmative answer to a conjecture of Lan-Sheng-Zuo \cite{LanShengZuo12ii}\footnotemark[1]. In addition, we prove a tensor product theorem for strongly semistable Higgs bundles with $p$ satisfying some bounds (Theorem \ref{TensorTheorem}). From this we reprove a tensor theorem for semistable Higgs bundles on the condition that the Lan-Sheng-Zuo conjecture holds (Corollary \ref{TensorStableBundle}).
\end{abstract}
\maketitle

\section{Introduction}

N. Hitchin \cite{Hitchin87}, K. Corlette \cite{Corlette88}\cite{Corlette93} and C. Simpson \cite{Simpson92}\cite{Simpson97} have established a correspondence between semistable Higgs bundles and representations of the fundamental groups on arbitrary dimensional complex projective manifolds. In order to establish an analogous correspondence in positive characteristic, Lan-Sheng-Zuo \cite{LanShengZuo12ii} introduced intermediate notions \emph{strongly semistable Higgs bundles} and \emph{quasi-periodic Higgs bundles} between semistable Higgs bundles and representations of algebraic fundamental groups, where strongly semistable Higgs bundles is a generalization of the notion of strongly semistable vector bundles in the sense of Lange-Stuhler \cite{LangeStuhler77}. They constructed a functor from the category of strongly semistable Higgs bundles with trivial Chern classes to the category of crystalline representations of fundamental groups, and showed that any rank $2$ nilpotent semistable Higgs bundle is a strongly semistable Higgs bundle (cf. \cite{LanShengZuo12ii}). In addition, they made the following conjecture.
\begin{Conjecture*}\label{Conjecture}
Any nilpotent semistable Higgs bundle of exponent less than the characteristic of base field is a strongly semistable Higgs bundle.
\end{Conjecture*}

In this paper, we show that any rank $3$ nilpotent semistable Higgs bundle is a strongly semistable Higgs bundle. This gives an affirmative answer to the conjecture of Lan-Sheng-Zuo \cite{LanShengZuo12ii} in the rank $3$ case.
\begin{Theorem}[Theorem \ref{Thm:StronglySemistableHiggsBundle}]
Let $k$ be an algebraically closed field of characteristic $p>2$, $X$ a smooth projective curve over $k$. Let $(E,\theta)$ be a rank $3$ nilpotent semistable Higgs bundle on $X$. Then $(E,\theta)$ is a strongly semistable Higgs bundle.
\end{Theorem}
In fact, the theorem above is also true in higher dimension since the proof is valid in the higher dimension case. With the results of Lan-Sheng-Zuo \cite{LanShengZuo12ii}, the theorem above implies that we can construct crystalline representations of fundamental groups associated to rank $3$ nilpotent semistable Higgs bundles with trivial Chern classes.

By the equivalence of strongly semistable Higgs bundles and quasi-periodic Higgs bundles over an algebraic closure of finite fields, we can prove the following tensor product theorem for strongly semistable Higgs bundles with trivial Chern classes.

\begin{Theorem}[Theorem \ref{TensorTheorem}]
Let $k$ be the algebraic closure of finite fields of characteristic $p>0$, and $X$ a smooth projective curve over $k$. Let $(E_1,\theta_1)$ and $(E_2,\theta_2)$ be strongly semistable Higgs bundles on $X$ with degree $0$. Suppose that $\rk(E_1)+\rk(E_2)\leq p+1$. Then the tensor product $(E_1\otimes E_2,\theta_1\otimes 1+1\otimes \theta_2)$ is also a strongly semistable Higgs bundle.
\end{Theorem}

Professor Kang Zuo explained to me that the idea in the proof is a flavor of a characteristic $p$ analog to the complex number field case. Over the complex number field, S. K. Donaldson, N. Hitchin and C. Simpson showed that a Higgs bundle is poly-stable if and only if it carries a Yang-Mills-Higgs metric (for semistable Higgs bundles one has the so-called approximated Yang-Mills-Higgs metric). If the degree of this vector bundle is zero, then one shows that the tensor product of Yang-Mills-Higgs metrics is again a Yang-Mills-Higgs metric. In this way, one proves the above theorem over the complex number field. For the characteristic $p$ case, the solution of Lan-Sheng-Zuo conjecture just means that any semistable Higgs bundle in characteristic $p$ of degree zero carries a characteristic $p$ Yang-Mills-Higgs metric. The above theorem can be interpreted as the tensor product of Yang-Mills-Higgs metrics in characteristic $p$ is again a Yang-Mills-Higgs metric in characteristic $p$.

If the Lan-Sheng-Zuo conjecture is true, one can reprove a tensor theorem for semistable Higgs bundles (cf. \cite[Theorem 8.16]{BalajiParameswaran11}).

\begin{Corollary}[Corollary \ref{TensorStableBundle}]
Suppose that Lan-Sheng-Zuo conjecture is true, i.e. any nilpotent semistable Higgs bundle of exponent less than the characteristic of base field is strongly Higgs semistable. If $(E_1,\theta_1)$ and $(E_2,\theta_2)$ are nilpotent semistable Higgs bundles on $X$ with degree $0$ and $\rk(E_1)+\rk(E_2)\leq p+1$, then the tensor product $(E_1\otimes E_2,\theta_1\otimes 1+1\otimes \theta_2)$ is also Higgs semistable.
\end{Corollary}

\section{Preliminary}

Let $X/\mathbb{C}$ be a smooth projective scheme over the complex number field, $(E,\nabla)$ a de Rham bundle on $X$ which satisfies suitable conditions. C. Simpson \cite{Simpson92} constructed a Higgs bundle $(E',\theta)$ associated to $(E,\nabla)$. Later A. Ogus and V. Vologodsky \cite{OgusVologodsky07} established a correspondence between category of de Rham bundles and category of Higgs bundles in positive characteristic using the theory of Azumaya algebra and certain universal algebra. Lan-Sheng-Zuo \cite{LanShengZuo12i} gave a explicit construction of inverse Cartier transform for nilpotent Higgs bundles of exponent $\leq p-1$ which is equivalent to the correspondence in \cite{OgusVologodsky07}.

Now, let's first recall the definitions of strongly semistable Higgs bundle and quasi-periodic Higgs bundle. Let $k$ be an algebraically closed field of characteristic $p>0$, $X$ a smooth projective curve of genus $g$ over $k$, and $F:X\rightarrow X$ the absolute Frobenius morphism. A \emph{Higgs-de Rham sequence} over $X$ is a sequence of form
$$
\xymatrix{
    &  (H_0,\nabla_0)\ar[dr]^{Gr_{Fil_0}}       &&  (H_1,\nabla_1)\ar[dr]^{Gr_{Fil_1}}  \\
 (E_0,\theta_0) \ar[ur]^{C^{-1}}  & &     (E_1,\theta_1) \ar[ur]^{C^{-1}}&&\ldots      }
$$
where $C^{-1}$ is the inverse Cartier transform, $Fil_i$ is a decreasing filtration on $H_i$ satisfies Griffiths transversality, $(H_i,\nabla_i)=C^{-1}(E_i,\theta_i)$ is a de Rham bundle and $(E_i,\theta_i)=Gr_{Fil_{i-1}}(H_{i-1},\nabla_{i-1})$ is a Higgs bundle on $X$.

\begin{Definition}\label{Def:StronglySemistableBundle}
A Higgs bundle $(E,\theta)$ is called \emph{strongly Higgs semistable} if it appears in the leading term of a Higgs-de Rham sequence whose
Higgs terms $(E_i,\theta_i)$s are all Higgs semistable.
\end{Definition}

\begin{Remark}
The definition of strongly Higgs semistable has been modified in the new version (\cite[Definition 2.2]{LanShengZuo13ii}) of the Definition 2.1 in \cite{LanShengZuo12ii}. But this does not affect our proof below.
\end{Remark}

\begin{Definition}\label{Def:Quasi-PeriodicBundle}
A Higgs bundle $(E,\theta)$ is called \emph{quasi-periodic} if it appears in the leading term of a quasi-periodic Higgs-de Rham sequence, i.e., it becomes periodic after a nonnegative integer.
\end{Definition}

\begin{Lemma}\label{Lem:Higgs-DeRham}
Let $k$ be an algebraically closed field, $X$ a smooth projective curve over $k$. Let $(H,\nabla)$ be a $\nabla$-semistable flat bundle on $X$, $H'$ a subsheaf of $H$ with $\mu(H')>\mu(H)$. Then the homomorphism $H'\rightarrow (H/H')\otimes_{\Ox_X}\Omg^1_X$ is non-trivial.
\end{Lemma}
\begin{proof}
Suppose that $H'$ is invariant under $\nabla$, then $\mu(H')\leq\mu(H)$ by the $\nabla$-semistability of $(H,\nabla)$. A contradiction.
\end{proof}

\begin{Remark}
Suppose that the characteristic of $k$ is $p>0$. Let $(E,\theta)$ be a nilpotent semistable Higgs bundle of exponent $\leq p-1$ on $X$ and $(H,\nabla):=C^{-1}(E,\theta)$. Then $(H,\nabla)$ is $\nabla$-semistable by \cite[Theorem 2.8]{OgusVologodsky07}.
\end{Remark}

\section{Rank $3$ semistable Higgs Bundles are Strongly Higgs semistable}

In this section, unless otherwise explicitly declared, we have the following notations and assumptions:

Let $k$ be an algebraically closed field of characteristic $p>0$, $X$ a smooth projective curve of genus $g\geq 0$ over $k$, and $(E,\theta)$ a rank $3$ nilpotent semistable Higgs bundle of exponent $\leq p-1$ on $X$. Denote $(H,\nabla):=C^{-1}(E,\theta)$,
where $C^{-1}$ is the inverse Cartier transform constructed in \cite{OgusVologodsky07} and \cite{LanShengZuo12i}. We would like to construct a Hodge filtration $\{Fil^i_H\}$ of $H$ with respect to the connection $\nabla$ such that the grading $Gr_{Fil^*_H}(H,\nabla)$ is a semistable Higgs bundle.

Suppose that $H$ is a semistable bundle, we can choose the trivial filtration $Fil_{tri}$ on $H$, then there is nothing need to prove. We assume that $H$ is not a semistable bundle. Let $S\subseteq H$ (resp. $H\twoheadrightarrow Q$) be the subbundle (resp. quotient bundle) of $H$ with maximal (resp. minimal) slope in the Harder-Narasimhan filtration of $H$.

\begin{Proposition}\label{Prop:S}
Let $Q':=H/S$, $\theta:S\hookrightarrow H\stackrel{\nabla}{\rightarrow}H\otimes_{\Ox_X}\Omg^1_X\twoheadrightarrow Q'\otimes_{\Ox_X}\Omg^1_X$ the natural homomorphism, and $\widetilde{S}$ the saturation of $\im(\theta)\otimes_{\Ox_X}T_X$ in $Q'$.
Suppose that $\rk(S)=1$ and $\mu_{max}(Q')\leq\mu(H)$. Then
\item[$(I)$] If $\mu(Q'/\widetilde{S})\leq\mu(H)$, then $(S\oplus\widetilde{S}\oplus(Q'/\widetilde{S}),\bigoplus\limits^2_{i=1}\theta^i)$ is a semistable Higgs bundle;
\item[$(II)$] If $\mu(Q'/\widetilde{S})\geq\mu(H)$, then $(S\oplus Q',\theta)$ is a semistable Higgs bundle.
\end{Proposition}
\begin{proof} For the case $\mu(H)=0$, the assumption of this proposition can be explained in the following diagram (The general case is similar to the slope zero case after rotate this diagram along the origin).

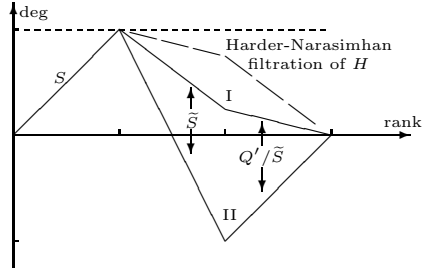
\begin{center}
\begin{picture}(150,100)
\linethickness{0.5pt}
\put(0,50){\vector(1,0){150}}
\put(0,0){\vector(0,1){100}}
\multiput(0,50)(40,0){4}{\line(0,1){2}}
\multiput(0,10)(0,40){3}{\line(1,0){2}}
\put(140,53){\tiny{rank}}
\put(2,95){\tiny{deg}}
\thinlines
\multiput(0,90)(4,0){30}{\line(1,0){2}} 
\put(0,50){\line(1,1){40}}
\put(15,70){\tiny{$S$}}
\multiput(44,89)(12,-3){3}{\line(4,-1){10}}
\multiput(80,80)(12,-9){3}{\line(4,-3){10}}
\put(80,75){\shortstack{\tiny{Harder-Narasimhan}\\\tiny{filtration of $H$}}}
\put(40,90){\line(4,-3){40}}
\put(80,60){\line(4,-1){40}}
\put(80,63){\tiny{I}}
\put(65,53){\tiny{$\widetilde{S}$}}
\put(67,60){\vector(0,1){8}}
\put(67,52){\vector(0,-1){9}}

\put(40,90){\line(1,-2){40}}
\put(80,10){\line(1,1){40}}
\put(79,18){\tiny{II}}
\put(85,40){\tiny{$Q'/\widetilde{S}$}}
\put(94,46){\vector(0,1){9}}
\put(94,38){\vector(0,-1){9}}
\end{picture}
\end{center}

$(I)$ Let $F$ be a nonzero Higgs subsheaf of $(S\oplus\widetilde{S}\oplus(Q'/\widetilde{S}),\bigoplus\limits^2_{i=1}\theta^i)$.

\textbf{Case I.i:} The composition $F\hookrightarrow S\oplus\widetilde{S}\oplus(Q'/\widetilde{S})\twoheadrightarrow Q'/\widetilde{S}$ is trivial, i.e. $F\subset S\oplus\widetilde{S}$.

Suppose that $\rk(F)=1$. Then the composition $F\hookrightarrow S\oplus\widetilde{S}\twoheadrightarrow\widetilde{S}$ is not trivial. Otherwise, we have $F\subset S$, this contradicts to the fact that $\theta^1:S\hookrightarrow\widetilde{S}\otimes_{\Ox_X}\Omg^1_X$ is injective by Lemma \ref{Lem:Higgs-DeRham}. Thus $\mu(F)\leq\mu(\widetilde{S})\leq\mu_{max}(Q')\leq\mu(H)$.

Suppose that $\rk(F)=2$. Then we can get that $\theta^2:\widetilde{S}\hookrightarrow(Q'/\widetilde{S})\otimes_{\Ox_X}\Omg^1_X$ is trivial and $\mu(Q'/\widetilde{S})=\mu(H)$. Thus $\mu(F)\leq\mu(S\oplus\widetilde{S})=\mu(H)$.

\textbf{Case I.ii:} The composition $F\hookrightarrow S\oplus\widetilde{S}\oplus(Q'/\widetilde{S})\twoheadrightarrow Q'/\widetilde{S}$ is not trivial.

Suppose that $\rk(F)=1$. Then the injection $F\hookrightarrow S\oplus\widetilde{S}\oplus(Q'/\widetilde{S})\twoheadrightarrow Q'/\widetilde{S}$ implies that  $\mu(F)\leq\mu(Q/\widetilde{S})\leq\mu(H)$.

Suppose that $\rk(F)=2$. Then $\rk(F\cap(S\oplus\widetilde{S}))=1$ and the composition $F\cap(S\oplus\widetilde{S})\hookrightarrow S\oplus\widetilde{S}\twoheadrightarrow\widetilde{S}$ is injective, since $\theta^1:S\hookrightarrow\widetilde{S}\otimes_{\Ox_X}\Omg^1_X$ is injective by Lemma \ref{Lem:Higgs-DeRham}.
Consider the commutative diagram
$$\xymatrix{
  0 \ar[r] & F\cap(S\oplus\widetilde{S})\ar[r]\ar@{^{(}->}[d] & F\ar[r]\ar@{^{(}->}[d] & F/(F\cap(S\oplus\widetilde{S}))\ar[r]\ar@{^{(}->}[d] & 0\\
  0 \ar[r] & S\oplus\widetilde{S}\ar[r]   & S\oplus\widetilde{S}\oplus(Q'/\widetilde{S})\ar[r]      & Q'/\widetilde{S}\ar[r]                          & 0.}$$
Then we have
\begin{eqnarray*}
\mu(F)&=&\frac{1}{2}[\mu(F\cap(S\oplus\widetilde{S}))+\mu(F/(F\cap(S\oplus\widetilde{S})))]\\
&\leq&\frac{1}{2}[\mu(\widetilde{S})+\mu(Q'/\widetilde{S})]\\
&<&\mu(H).
\end{eqnarray*}

Hence $(S\oplus\widetilde{S}\oplus(Q'/\widetilde{S}),\bigoplus\limits^2_{i=1}\theta^i)$ is a semistable Higgs bundle.

$(II)$ Let $F$ be a nonzero Higgs subsheaf of $(S\oplus Q',\theta)$.  Then the composition $F\hookrightarrow S\oplus Q'\twoheadrightarrow Q'$ is not trivial, since $\theta:S\hookrightarrow Q'\otimes_{\Ox_X}\Omg^1_X$ is not trivial by Lemma \ref{Lem:Higgs-DeRham}.

Suppose that $\rk(F)=1$. Then the injection $F\hookrightarrow S\oplus Q'\twoheadrightarrow Q'$ implies that $\mu(F)\leq\mu_{max}(Q')\leq\mu(H)$.

Suppose that $\rk(F)=2$. If the composition $F\hookrightarrow S\oplus Q'\twoheadrightarrow Q'$ is injective, then $\mu(F)\leq\mu(Q')<\mu(H)$. If the composition $F\hookrightarrow S\oplus Q'\twoheadrightarrow Q'$ is not injective, then $\rk(F\cap S)=\rk(F\cap Q')=1$. Consider the following commutative diagram
$$\xymatrix{
F\ar^-{\theta}[r]\ar@{^{(}->}[d]                & F\otimes_{\Ox_X}\Omg^1_X\ar@{^{(}->}[d]\\
S\oplus Q'\ar^-{\theta}[r]\ar@{->>}^{p_1}[d] & (S\oplus Q')\otimes_{\Ox_X}\Omg^1_X\\
S\ar@{^{(}->}^-{\theta}[r]                      & \widetilde{S}\otimes_{\Ox_X}\Omg^1_X\subset Q'\otimes_{\Ox_X}\Omg^1_X\ar@{^{(}->}^{i_2\otimes id_{\Omg^1_X}}[u].}$$
where $p_1:S\oplus Q'\rightarrow S$ is the natural projection and $i_2:Q'\hookrightarrow S\oplus Q'$ is the natural injection.
Thus $F\cap\widetilde{S}\neq 0$. Therefore, the saturations of $F\cap\widetilde{S}$ and $F\cap Q'$ in $Q'$ are the same, since $F\cap\widetilde{S}\subseteq F\cap Q'$ with $\rk(F\cap\widetilde{S})=\rk(F\cap Q')=1$. It follows that $F\cap Q'\subseteq\widetilde{S}$, since $F\cap\widetilde{S}\subseteq\widetilde{S}$ and $\widetilde{S}$ is a subbundle of $Q'$. Consider the commutative diagram
$$\xymatrix{
  0 \ar[r] & F\cap Q'\ar[r]\ar@{^{(}->}[d] & F\ar[r]\ar@{^{(}->}[d]  & F/(F\cap Q')\ar[r]\ar@{^{(}->}[d] & 0\\
  0 \ar[r] & Q'\ar[r]                      & S\oplus Q'\ar[r]     & S\ar[r]                      & 0.}$$
Then we have
\begin{eqnarray*}
\mu(F)&=&\frac{1}{2}[\mu(F\cap Q')+\mu(F/(F\cap Q'))]\\
&\leq&\frac{1}{2}[\mu(\widetilde{S})+\mu(S)]\\
&=&\frac{1}{2}[3\mu(H)-\mu(Q'/\widetilde{S})]\\
&\leq&\mu(H).
\end{eqnarray*}

Hence $(S\oplus Q',\theta)$ is a semistable Higgs bundle.
\end{proof}

\begin{Proposition}\label{Prop:Q}
Let $S':=\ker(H\twoheadrightarrow Q)$, $\theta:S'\hookrightarrow H\stackrel{\nabla}{\rightarrow}H\otimes_{\Ox_X}\Omg^1_X\twoheadrightarrow Q\otimes_{\Ox_X}\Omg^1_X$ the natural homomorphism, and $K:=\ker(\theta:S'\rightarrow Q\otimes_{\Ox_X}\Omg^1_X)$.
Suppose that $\rk(Q)=1$ and $\mu_{min}(S')\geq\mu(H)$. Then
\item[$(I)$] If $\mu(K)\leq\mu(H)$, then $(S'\oplus Q,\theta)$ is a semistable Higgs bundle;
\item[$(II)$] If $\mu(K)\geq\mu(H)$, then $(K\oplus(S'/K)\oplus Q,\bigoplus\limits^2_{i=1}\theta^i)$ is a semistable Higgs bundle.
\end{Proposition}

\begin{proof}
For the case $\mu(H)=0$, the assumption of this proposition can be explained in the following diagram (The general case is similar to the slope zero case after rotate this diagram along the origin).
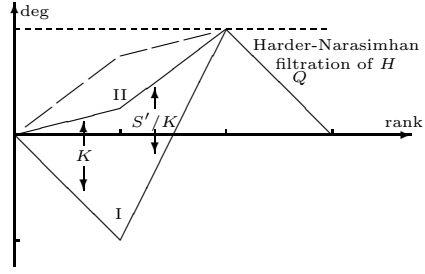
\begin{center}
\begin{picture}(150,100)
\linethickness{0.5pt}
\put(0,50){\vector(1,0){150}}
\put(0,0){\vector(0,1){100}}
\multiput(0,50)(40,0){4}{\line(0,1){2}}
\multiput(0,10)(0,40){3}{\line(1,0){2}}
\put(140,53){\tiny{rank}}
\put(2,95){\tiny{deg}}
\thinlines
\multiput(0,90)(4,0){30}{\line(1,0){2}} 
\put(80,90){\line(1,-1){40}}
\put(105,70){\tiny{$Q$}}
\multiput(4,53)(12,9){3}{\line(4,3){10}}
\multiput(40,80)(12,3){3}{\line(4,1){10}}
\put(90,75){\shortstack{\tiny{Harder-Narasimhan}\\\tiny{filtration of $H$}}}
\put(0,50){\line(4,1){40}}
\put(40,60){\line(4,3){40}}
\put(37,63){\tiny{II}}
\put(44,53){\tiny{$S'/K$}}
\put(53,60){\vector(0,1){8}}
\put(53,52){\vector(0,-1){9}}
\put(40,10){\line(1,2){40}}
\put(0,50){\line(1,-1){40}}
\put(38,18){\tiny{I}}
\put(23,40){\tiny{$K$}}
\put(26,46){\vector(0,1){9}}
\put(26,38){\vector(0,-1){9}}
\end{picture}
\end{center}

$(I)$ Let $F$ be a nonzero Higgs subsheaf of $(S'\oplus Q,\theta)$.

Suppose that $\rk(F)=1$. If the composition $F\hookrightarrow S'\oplus Q\twoheadrightarrow Q$ is not trivial, then $\mu(F)\leq\mu(Q)<\mu(H)$. If the composition $F\hookrightarrow S'\oplus Q\twoheadrightarrow Q$ is trivial, then $F\subseteq K$ since $K=\ker(S'\rightarrow Q\otimes_{\Ox_X}\Omg^1_X)$. Thus $\mu(F)\leq\mu(K)\leq\mu(H)$.

Suppose that $\rk(F)=2$. Then the composition $F\hookrightarrow S'\oplus Q\twoheadrightarrow Q$ is not trivial.
Consider the commutative diagram
$$\xymatrix{
  0 \ar[r] & F\cap S'\ar[r]\ar@{^{(}->}[d] & F\ar[r]\ar@{^{(}->}[d]  & F/(F\cap S')\ar[r]\ar@{^{(}->}[d] & 0\\
  0 \ar[r] & S'\ar[r]                      & S'\oplus Q\ar[r]     & Q\ar[r]                      & 0.}$$
Therefore, we have
\begin{eqnarray*}
\mu(F)&=&\frac{1}{2}[\mu(F\cap S')+\mu(F/(F\cap S'))]\leq\frac{1}{2}[\mu_{max}(S')+\mu(Q)]\\
&=&\frac{1}{2}[3\mu(H)-\mu_{min}(S')]\\
&\leq&\mu(H).
\end{eqnarray*}

Hence $(S'\oplus Q,\theta)$ is a semistable Higgs bundle.\\

$(II)$ Let $F$ be a nonzero Higgs subsheaf of $(K\oplus(S'/K)\oplus Q,\bigoplus\limits^2_{i=1}\theta^i)$.

Suppose that $\rk(F)=1$. If the composition $F\hookrightarrow K\oplus(S'/K)\oplus Q\twoheadrightarrow Q$ is not trivial, then $\mu(F)\leq\mu(Q)<\mu(H)$. If the composition $F\hookrightarrow K\oplus(S'/K)\oplus Q\twoheadrightarrow Q$ is trivial, then $F\subseteq K\subset K\oplus(S'/K)$ since $(S'/K)\rightarrow Q\otimes_{\Ox_X}\Omg^1_X$ is injective by Lemma \ref{Lem:Higgs-DeRham}. In this case, we must have $\mu(K)=\mu(H)$. Otherwise, if $\mu(K)>\mu(H)$, then $K\rightarrow(S'/K)\otimes_{\Ox_X}\Omg^1_X$ is injective by Lemma \ref{Lem:Higgs-DeRham}. This controdicts to $F\subseteq K$. Thus $\mu(F)\leq\mu(K)=\mu(H)$.

Suppose that $\rk(F)=2$. Then the composition $F\hookrightarrow K\oplus(S'/K)\oplus Q\twoheadrightarrow Q$ is not trivial, since $\theta^2:(S'/K)\hookrightarrow Q\otimes_{\Ox_X}\Omg^1_X$ is injective. Moreover, we have $$\mu(F\cap(K\oplus(S'/K)))\leq\mu(S'/K).$$ In fact, if $\mu(K)>\mu(H)$, then the homomorphism $\theta^1:K\hookrightarrow (S'/K)\otimes_{\Ox_X}\Omg^1_X$ is injective by Lemma \ref{Lem:Higgs-DeRham}. In this case, the composition $F\cap(K\oplus(S'/K))\hookrightarrow K\oplus(S'/K)\twoheadrightarrow S'/K$ is not trivial, then $\mu(F\cap(K\oplus(S'/K)))\leq\mu(S'/K)$. If $\mu(K)=\mu(H)$, then $$\mu(F\cap(K\oplus(S'/K)))\leq\mu_{max}(K\oplus(S'/K))=\mu(S'/K).$$
Consider the commutative diagram
$$\xymatrix{
  0 \ar[r] & F\cap(K\oplus(S'/K))\ar[r]\ar@{^{(}->}[d] & F\ar[r]\ar@{^{(}->}[d]  & F/(F\cap(K\oplus(S'/K)))\ar[r]\ar@{^{(}->}[d] & 0\\
  0 \ar[r] & K\oplus(S'/K)\ar[r]       & K\oplus(S'/K)\oplus Q\ar[r]       & Q\ar[r]             & 0.}$$
Therefore, we have
\begin{eqnarray*}
\mu(F)&=&\frac{1}{2}[\mu(F\cap(K\oplus(S'/K)))+\mu(F/(F\cap(K\oplus(S'/K))))]\\
&\leq&\frac{1}{2}[\mu(S'/K)+\mu(Q)]\\
&=&\frac{1}{2}[3\mu(H)-\mu(K)]\\
&\leq&\mu(H).
\end{eqnarray*}

Hence $(K\oplus(S'/K)\oplus Q,\bigoplus\limits^2_{i=1}\theta^i)$ is a semistable Higgs bundle.
\end{proof}

Combine Proposition \ref{Prop:S} with Proposition \ref{Prop:Q}, we can get the following theorem which is the main result of this section.

\begin{Theorem}\label{Thm:StronglySemistableHiggsBundle}
Let $k$ be an algebraically closed field of characteristic $p>2$, $X$ a smooth projective curve over $k$. Let $(E,\theta)$ be a rank $3$ nilpotent semistable Higgs bundle on $X$ and $(H,\nabla):=C^{-1}(E,\theta)$. Then there is a Hodge filtration $\{Fil^*_H\}$ of $(H,\nabla)$ such that the grading $Gr_{Fil^*_H}(H,\nabla)$ is a semistable Higgs bundle. In particular, $(E,\theta)$ is a strongly semistable bundle.
\end{Theorem}
\begin{proof}
Assume that $H$ is semistable, we can choose the trivial filtration of $(H,\nabla)$. Otherwise, it is easy to check that $H$ must satisfy one of the following cases
\begin{itemize}
       \item[$(i)$] $\rk(S)=1$ and $\mu_{max}(H/S)\leq\mu(H)$;
       \item[$(ii)$] $\rk(Q)=1$ and $\mu_{min}(\ker(H\twoheadrightarrow Q))\geq\mu(H)$;
\end{itemize}
where $S\subseteq H$ (resp. $H\twoheadrightarrow Q$) is the subbundle (resp. quotient bundle) with maximal (resp. minimal) slope in the Harder-Narasimhan filtration of $H$.

In the case $(i)$, we use the notation as Proposition \ref{Prop:S}. If $\mu(Q'/\widetilde{S})\leq\mu(H)$, we choose the Hodge filtration $$0\subsetneq S\subsetneq\ker(H\twoheadrightarrow Q'\twoheadrightarrow(Q'/\widetilde{S}))\subsetneq H.$$
If $\mu(Q'/\widetilde{S})\geq\mu(H)$, we choose the Hodge filtration $$0\subsetneq S\subsetneq H.$$
Then their gradings are semistable Higgs bundles by Proposition \ref{Prop:S}.

In the case $(ii)$, we use the notation as Proposition \ref{Prop:Q}. If $\mu(K)\leq\mu(H)$, we choose the Hodge filtration
$$0\subsetneq S'\subsetneq H.$$
If $\mu(K)\geq\mu(H)$, we choose the Hodge filtration
$$0\subsetneq K\subsetneq S'\subsetneq H.$$
Then their gradings are semistable Higgs bundles by Proposition \ref{Prop:Q}.

Since the grading $Gr_{Fil^*_H}(H,\nabla)$ is a system of Hodge bundle of rank $3$, it is a nilpotent semistable Higgs bundle of exponent $\leq 2$. Then we can construct inductively a Higgs-de Rham sequence with leading term $(E,\theta)$ such that every Higgs bundle term is Higgs semistable. Hence $(E,\theta)$ is a strongly semistable bundle.
\end{proof}

\section{Tensor Product of Strongly semistable Higgs Bundles}

In characteristic $0$ case, the tensor product of two semistable vector bundles is still semistable by Kobayashi-Hitchin correspondence. The Kobayashi-Hitchin correspondence in the setting of Higgs bundles has been generalized in a number of ways starting with the far reaching one by C. Simpson \cite{Simpson92} and \cite{Simpson94}. Thus the Higgs semi-stability of tensor product of two semistable Higgs bundles follows from the Higgs bundles version of the Kobayashi-Hitchin correspondence.

In the characteristic $p>0$ case, tensor product of two semistable vector bundles need not be semistable. However,
S. Ilangovan, V. B. Mehta, A. J. Parameswaran \cite{IlangovanMehtaParameswaran03} showed that if $E_1$ and $E_2$ are semistable vector bundles with $\rk(E_1)+\rk(E_2)\leq p+1$, then $E_1\otimes E_2$ is still a semistable vector bundle. Later, V. Balaji, A. J. Parameswaran \cite{BalajiParameswaran11} generalized the theorem of Ilangovan-Mehta-Parameswaran to the Higgs bundles case, and proved the following tensor product theorem for semistable Higgs bundles.
\begin{Theorem}\cite[Theorem 8.16]{BalajiParameswaran11}
Let $k$ be an algebraically closed field of characteristic $p>0$, $X$ a smooth projective curve over $k$. Let $(E_1,\theta_1)$ and $(E_2,\theta_2)$ be semistable Higgs bundles on $X$ with $\det(E_i)\cong\Ox_X$, $i=1,2$. Suppose that $\rk(E_1)+\rk(E_2)\leq p+1$. Then the tensor product $(E_1\otimes E_2,\theta_1\otimes 1+1\otimes \theta_2)$ is also a semistable Higgs bundle.
\end{Theorem}

The approach of V. Balaji, A. J. Parameswaran \cite{BalajiParameswaran11} is very complicated. In this section, we will give another more simple approach to prove a suitable modification of tensor product theorem for strongly semistable Higgs bundles.

In this section, unless otherwise explicitly declared, $k$ is an algebraic closure of finite fields of characteristic
$p>0$, and $X$ a smooth projective curve over $k$.

\begin{Lemma}\cite[Theorem 2.5]{LanShengZuo12ii}\label{Stongly-Quasiperiodic}
A Higgs bundle $(E,\theta)$ with trivial Chern classes is quasi-periodic if and only if $(E,\theta)$ is strongly Higgs semistable.
\end{Lemma}

One can easily induce the following tensor theorem for strongly semistable bundle by the equivalence of strongly semistable Higgs bundles and quasi-periodic Higgs bundles.

\begin{Theorem}\label{TensorTheorem}
Let $k$ be the algebraic closure of finite fields of characteristic $p>0$, and $X$ a smooth projective curve over $k$. Let $(E_1,\theta_1)$ and $(E_2,\theta_2)$ be strongly semistable Higgs bundles on $X$ with degree $0$. Suppose that $\rk(E_1)+\rk(E_2)\leq p+1$. Then the tensor product $(E_1\otimes E_2,\theta_1\otimes 1+1\otimes \theta_2)$ is also a strongly semistable Higgs bundle.
\end{Theorem}
\begin{proof}
Since the exponent of a nilpotent Higgs bundle less than its rank. Therefore, $\rk(E_1)+\rk(E_2)\leq p+1$ implies that $(E_1,\theta_1)\otimes(E_2,\theta_2)$ and $(E_i,\theta_i)(i=1,2)$ are nilpotent Higgs bundles of exponent $\leq p-1$. By Lemma \ref{Stongly-Quasiperiodic}, we have quasi-periodic Higgs-de Rham sequences $\{(E^{(i)}_j,\theta^{(i)}_j),(H^{(i)}_j,\nabla^{(i)}_j,Fil^{(i)}_j)\}_{j\in\Z_{\geq 0}}$ such that $(E^{(i)}_0,\theta^{(i)}_0)\cong(E_i,\theta_i)$, $i=1,2$.
Since inverse Cartier transform $C^{-1}$ preserve tensor structures, we have $$C^{-1}((E^{(1)}_0,\theta^{(1)}_0)\otimes(E^{(2)}_0,\theta^{(2)}_0))\cong C^{-1}(E^{(1)}_0,\theta^{(1)}_0)\otimes C^{-1}(E^{(2)}_0,\theta^{(2)}_0)=(H^{(1)}_0,\nabla^{(1)}_0)\otimes(H^{(2)}_0,\nabla^{(2)}_0).$$
Then there is a Hodge filtration $Fil_0$ on $(H^{(1)}_0,\nabla^{(1)}_0)\otimes(H^{(2)}_0,\nabla^{(2)}_0)$ such that $$Gr_{Fil_0}((H^{(1)}_0,\nabla^{(1)}_0)\otimes(H^{(2)}_0,\nabla^{(2)}_0))\cong Gr_{Fil^{(1)}_0}(H^{(1)}_0,\nabla^{(1)}_0)\otimes Gr_{Fil^{(2)}_0}(H^{(2)}_0,\nabla^{(2)}_0).$$
Thus we have $Gr\circ C^{-1}((E^{(1)}_0,\theta^{(1)}_0)\otimes(E^{(2)}_0,\theta^{(2)}_0))\cong(E^{(1)}_1,\theta^{(1)}_1)\otimes(E^{(2)}_1,\theta^{(2)}_1)$.
Hence we can construct inductively a Higgs-de Rham sequence with Higgs bundles terms $\{(E^{(1)}_j,\theta^{(1)}_j)\otimes(E^{(2)}_j,\theta^{(2)}_j)\}_{j\in\Z_{\geq 0}}$. As $\{(E^{(i)}_j,\theta^{(i)}_j),(H^{(i)}_j,\nabla^{(i)}_j,Fil^{(i)}_j)\}_{j\in\Z_{\geq 0}}(i=1,2)$ are quasi-periodic Higgs-de Rham sequences,
then $\{(E^{(1)}_j,\theta^{(1)}_j)\otimes(E^{(2)}_j,\theta^{(2)}_j)\}_{j\in\Z_{\geq 0}}$ is also a quasi-periodic Higgs bundles sequence. Hence $$(E_1,\theta_1)\otimes(E_2,\theta_2)=(E_1\otimes E_2,\theta_1\otimes 1+1\otimes \theta_2)$$
is a strongly semistable Higgs bundles by Lemma \ref{Stongly-Quasiperiodic}.
\end{proof}

As a consequence from the tensor product theorem for strongly semistable Higgs bundles, one can reprove the tensor product theorem for semistable Higgs bundles if the Lan-Sheng-Zuo conjecture is true.
\begin{Corollary}\label{TensorStableBundle}
Suppose that Lan-Sheng-Zuo conjecture is true, i.e. any nilpotent semistable Higgs bundle is strongly Higgs semistable. If $(E_1,\theta_1)$ and $(E_2,\theta_2)$ are nilpotent semistable Higgs bundles on $X$ with degree $0$ and $\rk(E_1)+\rk(E_2)\leq p+1$, then the tensor product $(E_1\otimes E_2,\theta_1\otimes 1+1\otimes \theta_2)$ is also Higgs semistable.
\end{Corollary}

\noindent\textbf{Acknowledgments:} I did the first part of this work during my stay at Johannes Gutenberg Universit\"{a}t Mainz, Germany, as a visiting fellow in 2012. I would like to express my hearty thanks to Professor Kang Zuo, who introduced this subject to me and told me the idea how to prove the Theorem \ref{TensorTheorem}. This paper would not be possible without his help and discussions. Thanks are also due to both Professor Mao Sheng and Gui-Tang Lan for many helps and effective discussions. I also wish to thank the hospitality of Institut f\"{u}r Mathematik, Johannes Gutenberg Universit\"{a}t Mainz, Germany.

\end{document}